\documentclass[11pt]{amsart}
\usepackage{graphicx}
\usepackage{enumerate}
\newtheorem{thm}{Theorem}[section]
      \newtheorem{cor}[thm]{Corollary}

      \newtheorem{lemma}[thm]{Lemma}

      \newtheorem{rem}[thm]{Remark}

\begin{document}

\title[Xer recombination on catenanes]
{Rational tangle surgery and \\ Xer recombination on catenanes}
\author{Isabel K. Darcy, Kai Ishihara, Ram K. Medikonduri, and Koya Shimokawa}
\address{Department of Mathematics \& Applied Mathematical and Computational
Sciences,
University~of~Iowa,
14 MLH
Iowa City, IA 52242, 
USA (Isabel K. Darcy)}
\email{idarcy.math@gmail.com}
\address{Department of Mathematics,
Imperial College London,
London SW7 2AZ, UK (Kai Ishihara)}
\email{k.ishihara@imperial.ac.uk}
\address{Pyxis Solutions, New York, NY 10004, USA (Ram K. Medikonduri)}
\email{medukishore@gmail.com}
\address{Department of Mathematics,
                      Saitama University,
                      Saitama, 338-8570, Japan (Koya Shimokawa)}
\email{kshimoka@rimath.saitama-u.ac.jp}
%


\begin{abstract}
The protein recombinase can change the knot type of circular DNA.  
The action of a recombinase converting one knot into another knot is normally mathematically modeled by band surgery.  Band surgeries on a $2$-bridge knot $N(\frac{4mn-1}{2m})$ yielding a $(2,2k)$-torus link are characterized.  We apply this and other rational tangle surgery results to analyze Xer recombination on DNA catenanes using the tangle model for protein-bound DNA.  
\end{abstract}

\maketitle

\section{Introduction}

A tangle consists of arcs properly embedded in a three dimensional ball. A protein-DNA complex can be regarded as a tangle in which the protein complex is considered as a three dimensional ball and the DNA within the complex as arcs \cite{ES1, SECS}.
In general it is very hard to identify the arrangement of DNA within a
protein-DNA complex. Electron micrographs, AFM (Atomic Force Microscopy)
images, and crystalline structures do not give clear enough data to
address this problem for large molecules. Thus the tangle model is
frequently used  in analyzing the topology of protein-bound DNA \cite{BV, DS1, ES1, ES2, VS}.

The proteins we are interested in are ones which are involved in recombination. 
{\em Site-specific recombination} is a process in which specific target sequences on each of two DNA segments are exchanged. These specific sites are called {\em recombination sites}. Proteins that carry out these recombination reactions are called {\em recombinases}. Site-specific recombination reactions are involved in a variety of biological processes including transposition of DNA, integration into host chromosomes, and gene regulation.

During recombination, the topology of circular DNA can change forming knots and links.  The local action of a recombinase has been modeled by the mathematical operation of   
a band surgery or 
a rational tangle surgery.
In this paper we will characterize 
such surgeries
on a $2$-bridge knot $N(\frac{4mn-1}{2m})$ yielding  a $(2,2k)$-torus link.
The class of knots $N(\frac{4mn-1}{2m})$ includes the family of twist knots which frequently appear in biological reactions.  It also includes knots which are believed to be the products of Xer recombination when Xer acts on  $(2,2k)$-torus links.  

In section \ref{tangles}, we give mathematical preliminaries.  In section \ref{result}, we state the main result which we apply  to Xer recombination in section \ref{xer}.  We prove our main result in section \ref{proof}. In section \ref{non-band} we consider non-band rational tangle surgery cases related to Xer recombination.
In section \ref{summary} we summarize our results and briefly discuss the software TopoIce-R \cite{Darcy15072006} within Knotplot \cite{SchPhD} which implements the results of Theorem \ref{Thm:rts}.

\section{Tangles and 2-bridge knots and links}\label{tangles}
 
Let $T$, $T_1$, and $T_2$ be $2$-string tangles. 
The knot (or link) obtained by connecting the top two endpoints of $T$ and the bottom two endpoints by simple curves, as shown in Figure \ref{num}, is called the {\em numerator closure} and is denoted by $N(T)$.
From two tangles $T_1$ and $T_2$ a new tangle, called the {\em sum} of $T_1$ and $T_2$ and denoted by $T_1+T_2$, can be obtained by connecting two endpoints of $T_1$ to two endpoints of $T_2$ as shown in Figure \ref{tsum}. 
For more on tangles see \cite{Conway, KG, Mu1, SECS}.

\begin{figure}[htbp]
\begin{center}
\begin{minipage}{0.49\textwidth}
\begin{center}
\includegraphics[scale=0.5]{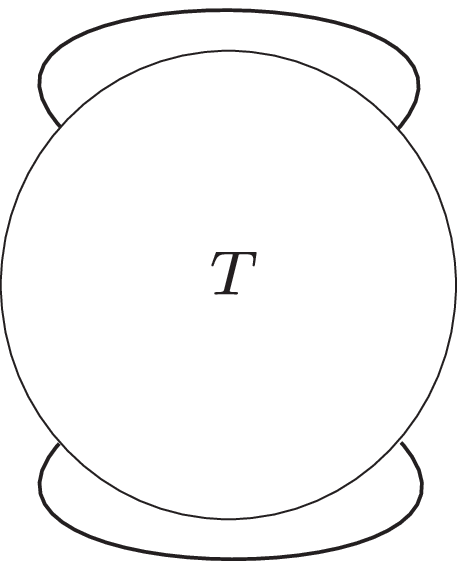}
\end{center}
\caption{The numerator $N(T)$.}
\label{num}
\end{minipage}
\begin{minipage}{0.5\textwidth}
\begin{center}
\includegraphics[scale=0.5]{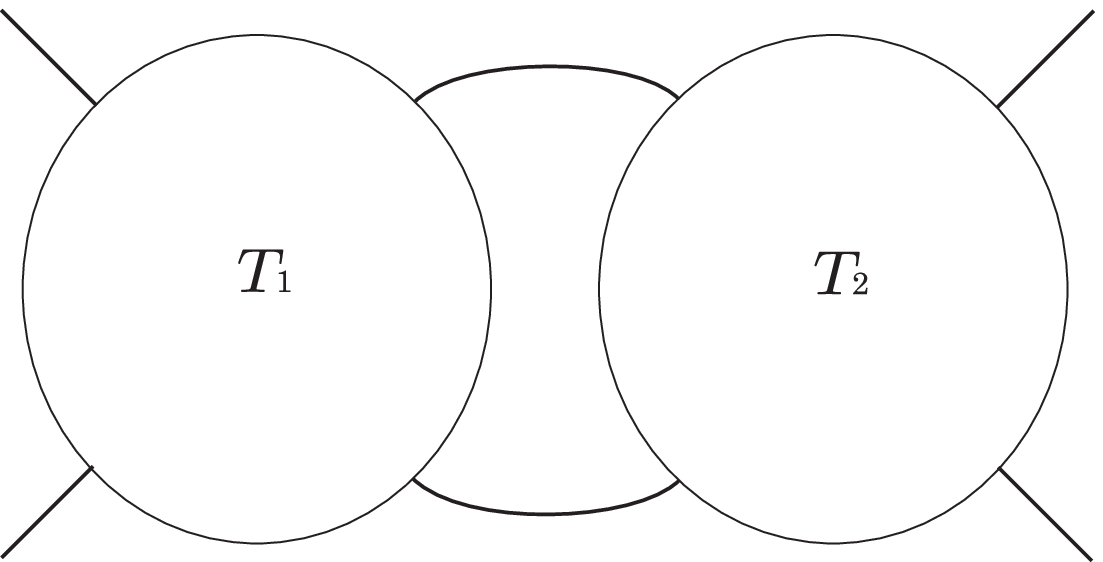}
\end{center}
\caption{The tangle sum $T_1+T_2$.}
\label{tsum}
\end{minipage}
\end{center}
\end{figure}

Let $c_1,\ldots,c_n$ be a sequence of integers. The {\em circle product} of a tangle $T$ and $(c_1,\ldots,c_n)$ is defined as shown in Figure \ref{cpr} and is denoted by $T\circ (c_1,\ldots,c_n)$ \cite{D_unor}.  A rational tangle is the circle product of a zero crossing tangle  and $(c_1,\ldots,c_n)$.  The zero crossing tangle contains two vertical strings if $n$ is even or two horizontal strings if $n$ is odd.  
Rational tangles are classified up to ambient isotopy fixing the boundary by their continued fraction $c_n + \frac{1}{c_{n-1} + \frac{1}{ ... + \frac{1}{c_1}}}$ \cite{Conway}.

\begin{figure}[htbp]
\begin{center}
\includegraphics[scale=0.5]{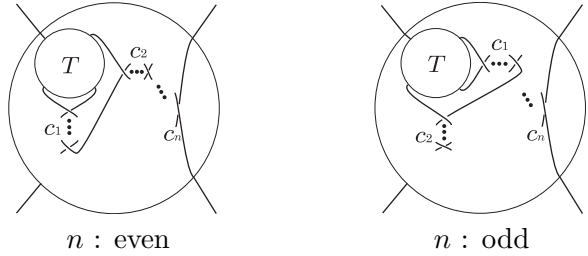}

$n$ : even\hspace{3.5cm}$n$ : odd
\end{center}
\caption{Circle products $T\circ (c_1,\ldots,c_n)$ of $T$ and $(c_1,\ldots,c_n)$: 
The vertical twists are left-handed if $c_i > 0$ and right-handed if $c_i < 0$ while the horizontal twists are right-handed if $c_i > 0$ and left-handed if $c_i < 0$.
}
\label{cpr}
\end{figure}

 The  knot/link, $N(\frac{a}{b})$ can be formed from the rational tangle $\frac{a}{b}$ via numerator closure. $N(\frac{a}{b})$ is a 2-bridge knot/link except when $|a| = 1$ in which case it is the unknot.  $N(\frac{a}{b})$ is a link if and only if $a$ is even in which case it is a 2-component link.
The 2-component link $N(2k)$ is called a $(2,2k)$-torus link. 
Given a projection of a link, the linking number of a 2-component link can be determined by calculating the sum over all crossings involving both components of the signed crossing number.  

Let $P$ and $R$ be rational tangles. The operation which transforms a knot (or link) by replacing $P$ with $R$ 
is called a {\em rational tangle surgery}, and we will refer to this as a {\em $(P, R)$} move. See Figure 
\ref{rts}.   A $(P, R)$ move is said to be {\em equivalent} to a $(P', R')$ move if and only if  for every pair of knots, 
$K_1$ and $K_2$,  there exists $U$ satisfying the system of tangle equations $N(U + P) = K_1$, $N(U + R) = K_2$ 
if and only if  there exists $U'$ satisfying the system of tangle equations $N(U' + P') = K_1$, $N(U' + R') = 
K_2$.  The following two theorems classify equivalent  $(P, R)$ moves when $P$ and $R$ are rational tangles.

\begin{figure}[htbp]
\begin{center}
\includegraphics[scale=0.5]{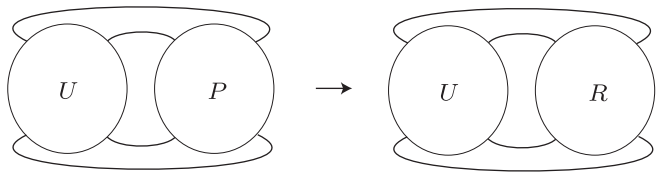}
\end{center}
\caption{A rational tangle surgery:  $P$ and $R$ are rational $2$-string tangles.  $U$ is a $2$-string tangle.}
\label{rts}
\end{figure}

\begin{thm}\cite{D_unor}
 A $(0, \frac{t}{w})$ move is equivalent to a $(0, \frac{c}{d})$ move if and only if $\frac{c}{d} = \frac{t}{w - ht}$ for some
$h$.  Moreover, $N(U + \frac{0}{1}) = K_1$ and $N(U + \frac{t}{w} ) = K_2$ if and only if $N([U \circ (h, 0)] + \frac{0}{1}) =
K_1$ and $N([U \circ (h, 0)] + \frac{t}{w - ht} ) = K_2$.
 \label{t-equival}
 \end{thm}

 \begin{thm}\cite{D_unor} 
An $(\frac{f_1}{g_1}, \frac{f_2}{g_2})$ move is equivalent to a $(0, \frac{t}{w})$ move if and only
if there exists $e_1$ and $i_1$ such that $g_1e_1 - f_1i_1 = 1$ and $\frac{t}{w} =
 \frac{g_1f_2 - g_2f_1}{e_1g_2 - i_1f_2}$ (or equivalently, $\frac{f_2}{g_2} = \frac{te_1 + wf_1}{ti_1 + wg_1})$
\label{t-ratequiv} \end{thm}

As discussed in \cite{D_unor}, any solution for $U$ to $N(U + P) = K_1$, $N(U + R) = K_2$ can be translated 
into a solution for $U'$ to $N(U' + 0) = K_1$, $N(U' + R') = K_2$ and vice versa.  Hence we will often focus on the $P = 0$ case.

Let $L$ be a link in $S^3$. 
Let $b : I\times I \to S^3$ be a band 
satisfying $b^{-1}(L)=I\times \partial I$, 
where $I=[0,1]$ is an interval. 
Let $L_b$ denote a link obtained by replacing $b(I\times \partial I)$ in $L$ 
with $b(\partial I \times I)$. For simplicity we denote $b(I\times I)$ by $b$. 
We say $L_b$ is obtained from $L$ by a {\em band surgery} along $b$.
  Note that band surgery is equivalent to rational tangle surgery where $P = 0$ and $R = \frac{1}{w}$ 
for some integer $w$.  If $L$ and $L_b$ have orientations which agree with each other except for the band $b$, 
the corresponding band surgery is said to be {\em coherent}. 
A $(0,\frac{1}{w})$ move can be considered as a band surgery.
See Figure \ref{br-surgery}.
Note that by choosing the orientations of $L$ and $L_b$, any band
surgery can be a coherent one with respect to those orientations.

\begin{figure}[htbp]
\begin{center}
%
\includegraphics[scale=0.6]{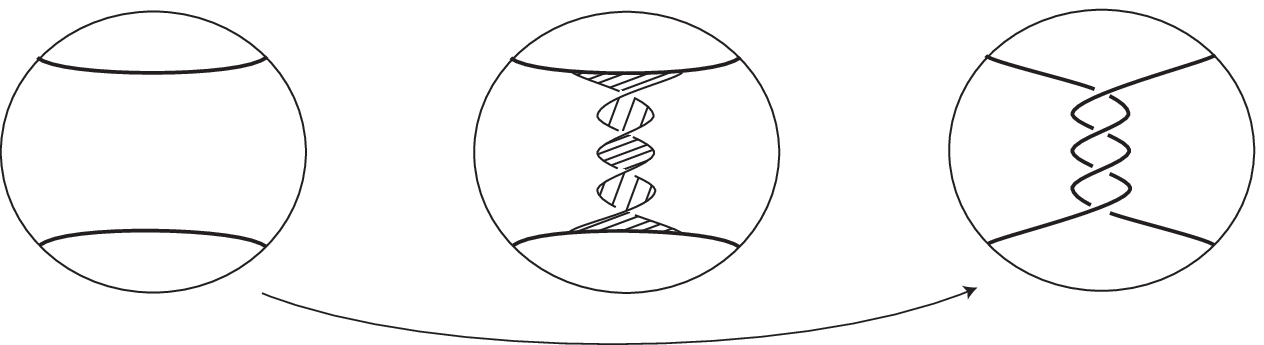}

$(0,-\frac{1}{4})$ move
\caption{A band surgery and a $(0,\frac{1}{w})$ move.}
\label{br-surgery}
\end{center}
\end{figure}

\section{Statement of Main Theorem:  Characterization of band surgery} \label{result}

In this section we state a result characterizing band surgeries
on a $(2,2k)$-torus link $N(2k)$ yielding a $2$-bridge knot $N(\frac{4mn-1}{2m})$.  

\begin{thm}\label{Thm:rts}
Suppose $N(U+0)=N(\frac{4mn-1}{2m})$ and $N(U+\frac{1}{w})= (2, 2k)$-torus link (*) 
and the rational tangle surgery corresponds to a 
coherent band surgery.
If the $(2, 2k)$-torus link has linking number $k$ where $|k| > 2$, then (*) has no solution. If the $(2, 2k)$-torus link has linking number $-k$, then 
one of the following holds $($see Figure \ref{4t} below$)$:
\begin{enumerate}
 \item[(1)] $k=m$ and $U=(\frac{4mn-1}{-w(4mn-1)+2m})$.
 \item[(2)] $k=n$ and $U=(\frac{4mn-1}{-w(4mn-1)+2n})$.
 \item[(3)] $k=m+n+1$ and $U=(\frac{-1}{2m+1}+\frac{-1}{2n+1})\circ(1,-(w+1),0)$.
 
 \rightline{ or $U=(\frac{-1}{2n+1}+\frac{-1}{2m+1})\circ(1,-(w+1),0)$.}
 \item[(4)] $k=m+n-1$ and $U=(\frac{-1}{2m-1}+\frac{-1}{2n-1})\circ(-1,-(w-1),0)$.
 
 \rightline{or $U=(\frac{-1}{2n-1}+\frac{-1}{2m-1})\circ(-1,-(w-1),0)$.}
\end{enumerate}
\end{thm}

 Note if the $(2, \pm4)$-torus link has linking number $\pm 2$ where the $\pm$ signs 
agree, then we do not know if (*) has a solution corresponding to a coherent 
band surgery. See subsection \ref{Sec:4cat}  
for a case where the $(2, 4)$-torus link has linking number $2$ and the product is the trefoil knot. 
However, if the $(2, \pm4)$-torus link has linking number $\mp 2$, then we 
can use Theorem \ref{Thm:rts} to determine all solutions (if any) to (*). 
 The Hopf 
link = the $(2, 2)$-torus link = the $(2, -2)$-torus link.  Thus Theorem \ref{Thm:rts} 
applies to the Hopf link no matter how it is oriented.  
When $k = 0$, $N(2k) = N(0) =$ unlink of two components.  Hence Theorem \ref{Thm:rts} also applies to the two component unlink.
 
 The case where $mn = 0$ (i.e. $N(\frac{4mn-1}{2m})$ is an unknot) was proved in \cite{HS}.  In section \ref{proof}, we prove Theorem \ref{Thm:rts} for the case where  $m, n \not= 0$.

\begin{figure}[htbp]
\begin{center}
\includegraphics[scale=0.8]{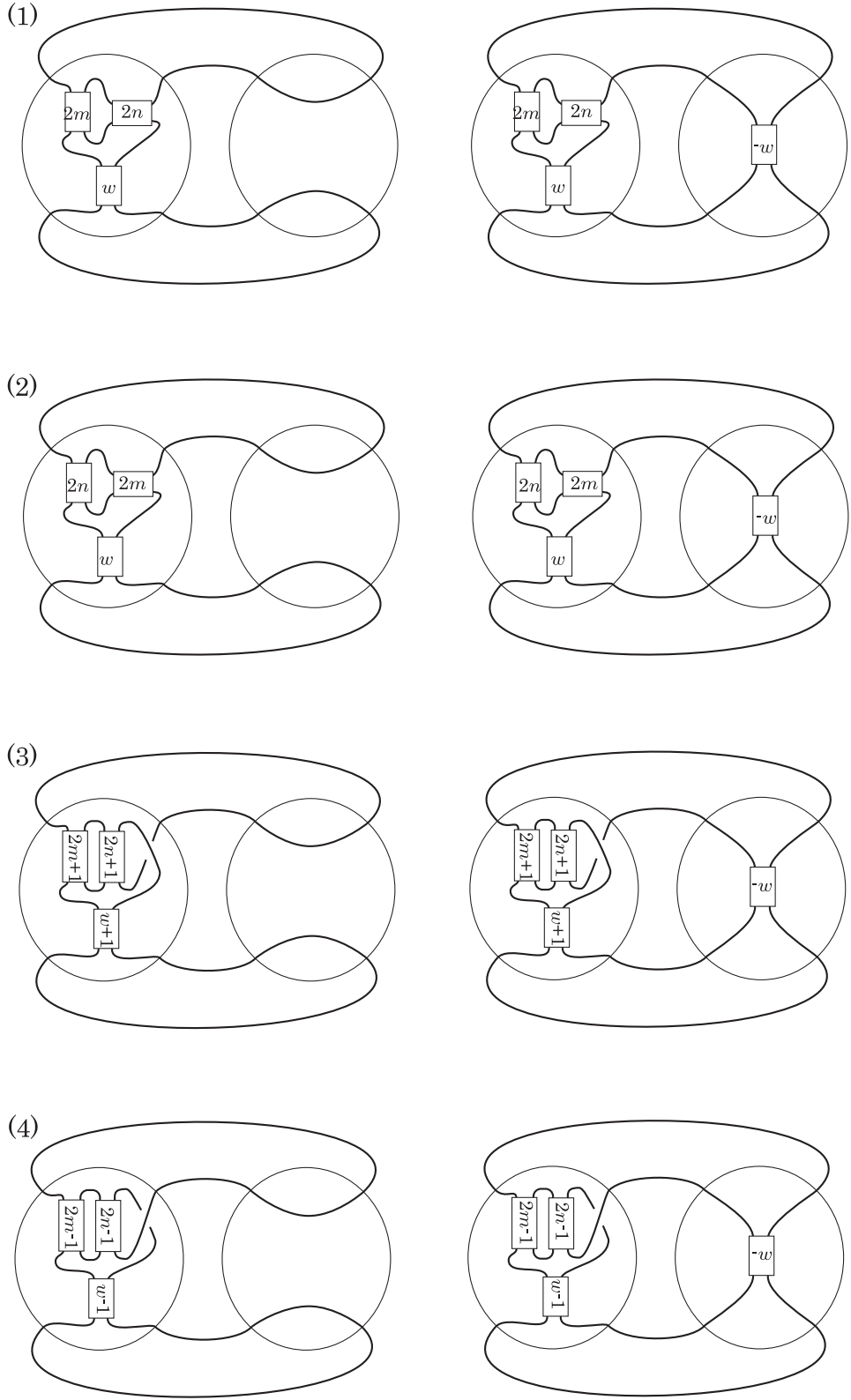}
\caption{Rational tangle surgeries from $N(\frac{4mn-1}{2m})$ to $N(2k)$.  A positive (negative) number inside a box corresponds to right (left) handed twists.  For example, $N(\frac{4mn - 1}{2m}) = (-2m, 2n, -w, 0)$ per left side of (1).}
\label{4t}
\end{center}
\end{figure}

\section{Results on Xer recombination on DNA catenanes}\label{xer}

Since recombinases normally bind to asymmetric recombination sites, the DNA sequence of the recombination sites 
can be used to orient these sites as shown in Figure \ref{directinverse}.   When the two sites are on the same 
component, their orientations on the circular DNA molecule can either agree (Figure \ref{directinverse}A) or 
disagree (Figure \ref{directinverse}B).  In the former case, we say that the two sites are directly repeated, 
while in the later case, we say that the two sites are inversely repeated. Observe that recombination on 
directly repeated sites normally results in a change in the number of components while for inversely repeated sites the number of components is normally preserved.

\begin{figure}[htbp]
\begin{center}
\includegraphics[width=13cm]{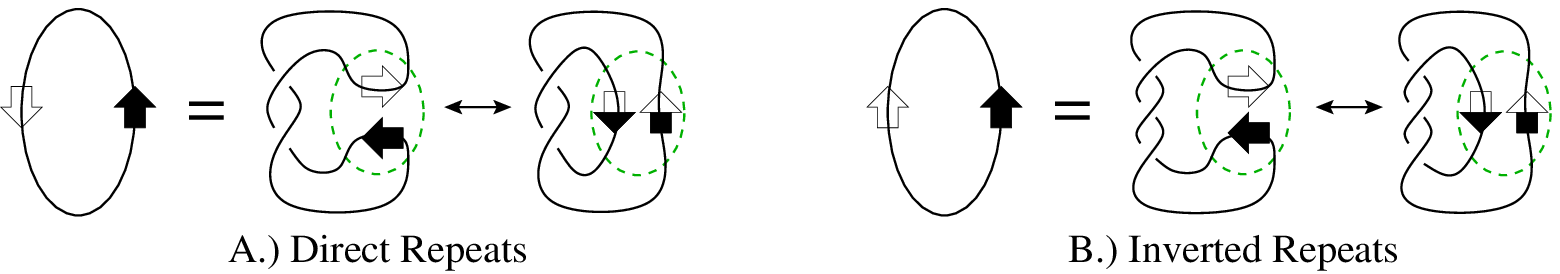}
\caption{A.) Recombination on directly repeated sites. B.) Recombination on inversely repeated sites.}
\label{directinverse}
\end{center}
\end{figure}

 When a site-specific recombinase binds to two segments of circular DNA, the DNA can be partitioned in 
different ways.  Sometimes one partitions the DNA into two tangles where one tangle represents the DNA bound by 
the protein while the other tangle represents the free DNA not bound by protein.  
  In other cases, we are interested in the local action of the DNA.  In this case, one tangle represents the two 
very short segments of DNA upon which the recombinase acts, breaking, exchanging and rejoining the DNA 
segments.  This local recombination reaction is modeled by replacing the tangle $P$ with the tangle $R$ (Figure 
\ref{rts}).    
    By taking a very small tangle ball around the DNA segments which are broken by the recombinase, we can assume 
the tangle $P$ does not contain any crossings.  Thus we can take the $P$ to be the $0$ tangle. 
    The remaining DNA configuration represented by the tangle $U$ is unchanged by the recombinase action (Figure \ref{rts}).  

Currently the local action of various recombinases have been modeled by $(0, \frac{1}{w})$ moves where $|w| = 
0, 1, 2$
(e.g. \cite{CSC, HGS, SECS}).  
Thus the local 
action of a recombinase is believed to be equivalent to a band surgery.  Moreover in some cases, the band surgery can be 
assumed to be coherent.
   When the sites are inversely repeated, recombination normally results in the inversion of one DNA segment 
with respect to the other one.  Thus this case does not correspond to a coherent band surgery.  However, when the 
sites are directly repeated or on different components, we can use these sites to orient the DNA.  In this 
case, the chemistry of the reaction normally requires that the recombination  correspond to a coherent band surgery.

\subsection{Xer recombination}
 When acting on circular DNA, recombinases can invert a DNA segment, delete a DNA segment or fuse together two 
DNA circles (Figure \ref{directinverse}). 
When two identical DNA circles fuse together into one larger DNA circle, the larger DNA circle contains two 
copies of the same DNA sequence (one from each of the identical circles).  The larger DNA circle made up of two  
copies of the same DNA sequence is called a dimer (di = two, mer = part). 
 This can occur to the genome of the bacteria E-coli.  When two circular genomes of E-coli fuse to form one 
larger circle containing two copies of the E-coli genome, the genome of the E-coli may not properly segregate 
when the host cell divides to form two new daughter cells.  Thus for proper segregation of its genome upon cell 
division, E-coli needs a method to change the dimer into two monomers.  This is the job of Xer 
recombinase.

The biochemistry of Xer recombination is very similar to recombinases such as Flp and Cre.  However its 
topology is very different.  Cre and Flp can produce a spectrum of knots and links when acting on circular DNA.  
They can invert, delete, or fuse together DNA.  Xer on the other hand produces a unique product dependent only 
on the starting configuration of the DNA (usually unknotted, but not always as seen below). Xer's function is 
to create two DNA circles from a larger DNA circle.  In other words, its function is to performs deletions.  
Xer recombinase does not need to invert a DNA segment, and it should not accidently fuse two DNA circles into a 
dimer.  But how does a protein know when it is acting on two segments from the same or different molecules or 
whether it is deleting or inverting a DNA segment?  Xer uses a topological filter to ensure that it only 
performs deletions when acting on unknotted circular DNA \cite{CBS}.
The term {\em topological filter} is used by  biologists to describe the mechanism in which a protein such as 
Xer  sets up a specific protein-DNA topology in order to select a particular reaction pathway -- in this case 
deletion instead of fusion or inversion.

When acting on unknotted circular DNA, Xer produces the $(2, 4)$-torus link with linking number $-2$ \cite{CBS}.  It is 
believed that Xer  uses the topological mechanism shown in Figure \ref{xer4cat}. The local action of Xer 
recombination is modeled by the small dashed circles: the $P = 0$ tangle is changed into the $R = -1$ tangle.  
The larger black circles in Figure \ref{xer4cat} denote the tangles modeling the entire protein-DNA complex. We will use $B$ and $E$ to represent these larger tangles  
modeling the protein-bound DNA before and 
after recombination, respectively.
Xer 
uses accessory proteins to trap three  DNA crossings. Thus this protein-DNA complex includes the Xer proteins, 
the accessory proteins and the three DNA crossings. It is modeled by the $B = -\frac{1}{3}$ tangle.  If 
the Xer binding sites are directly repeated, the three crossings  brings the Xer sites into proper conformation 
so that recombination can occur.  Thus the protein bound DNA configuration changes from the $B = -\frac{1}{3}$ 
tangle to the $E = -\frac{4}{3}$ tangle. 
We would need to change the orientation of one of the arrows in Figure \ref{xer4cat} if  the DNA sites were 
inversely repeated.  Xer cannot act on such a conformation and thus
inversion is prevented. 
The $B = 
-\frac{1}{3}$ tangle conformation also makes fusion of two unlinked DNA circles unlikely as it is biologically 
difficult to form the three crossings in the $B = -\frac{1}{3}$ tangle if the DNA comes from two unlinked circular DNA 
molecules.  
Thus deletion is normally preferred over both inversion and fusion.

\begin{figure}[htbp]
\begin{center}
\includegraphics[width=8cm]{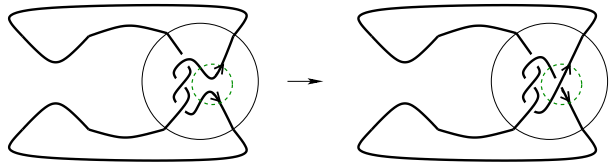}
\caption{Xer recombination on the unknot results in the $(2, 4)$-torus link}
\label{xer4cat}
\end{center}
\end{figure}

\subsection{Xer recombination from $(2,2k)$-torus link to $2k+1$
  crossing knot}
In \cite{BSC} Bath, Sherratt and Colloms studied 
Xer site-specific recombination on DNA catenanes whose 
link types were $(2,2k)$-torus links, $(k\geq 2)$.  Xer binding sites were placed on each component of these links, and these sites were used to orient the  $(2,2k)$-torus link substrates.  If these   $(2,2k)$-torus links had linking number $-k$ and if $k \geq 3$, then Xer  recombination can occur relatively efficiently 
and yields DNA knots with 
$2k+1$ crossings. That is, Xer recombination can result in the fusion of two DNA molecules.  
For example, Figure \ref{7_4} shows possible configurations of DNA before and after recombination in the case 
where $k=3$. Xer did not act efficiently when the $(2,2k)$-torus links had linking number $k$, nor did it act efficiently on $N(4)$ regardless of the orientation of the Xer binding sites.
  Recall that the topological filter used by Xer is suppose to prevent fusion from occurring.  Thus in order to  understand why Xer can perform fusion on some substrates but not others, we need to understand the topology of these reactions.  According to Theorem \ref{Thm:rts},  assuming Xer's local action corresponds to a coherent band surgery, it is not mathematically possible for Xer to act on a $(2,2k)$-torus links with  linking number $k \geq 3$  and produce the knot $N(\frac{4mn-1}{2m})$.  This corresponds nicely with
  Xer's inability to act when the $(2,2k)$-torus links had linking number $k$, $k \geq 2$. Theorem \ref{rts} only applies when the product is $N(\frac{4mn-1}{2m})$ and does not give us any information about other types of knots, but it is a good first step. 
That Xer can act on $(2,2k)$-torus links with linking number $-k$ when $k \geq 3$, but not on the $(2, 4)$ torus link is more intriguing.





To better understand Xer recombination, we will first focus on its local action. 
By Theorem \ref{t-equival}, a $(0,\frac{1}{w})$-move and $(0,\frac{1}{w'})$-move
are equivalent to each other for any integers $w$ and $w'$. 
 The local  action of Xer recombination has been modeled by replacing the $P = 0$-tangle with the $R = -1$-tangle.    
Thus we fix the integer $w$: $w=-1$ in Theorem \ref{Thm:Xer2k} and corollaries.
Recall that the chemistry of the reaction normally requires that the local action correspond to a coherent band surgery.

\begin{figure}[htbp]
\begin{center}
\includegraphics[scale=0.6]{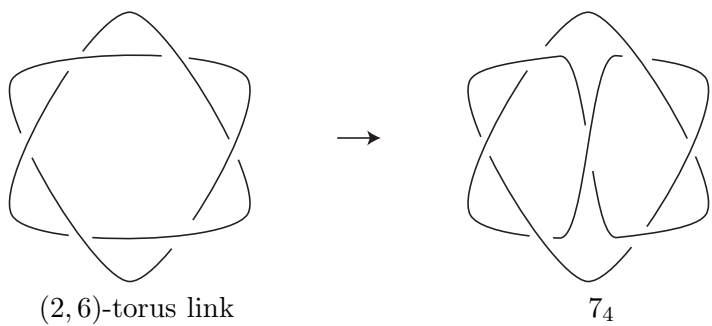}

\hspace{-10mm} $(2,6)$-torus link\hspace{47mm}$7_4$
\caption{Xer recombination from the $(2,6)$-torus link to the $7_4$ knot.}
\label{7_4}
\end{center}
\end{figure}



\begin{figure}[htbp]
\begin{center}
\includegraphics[scale=0.6]{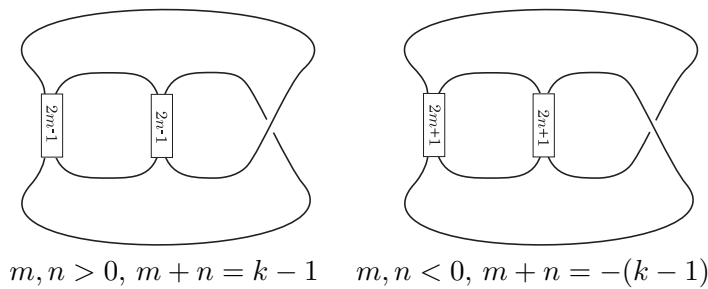}

$m,n>0$, $m+n=k-1$\hspace{5mm}$m,n<0$, $m+n=-(k-1)$
\caption{Knots with $2k+1$ crossings.}
\label{2k+1crossing}
\end{center}
\end{figure}

\begin{thm}\label{Thm:Xer2k}
Suppose $N(U+0)= 
N(2k)$ and $N(U+(-1))= N(\frac{4mn-1}{2m})$, where $k\ge 1$ and $N(\frac{4mn-1}{2m})$ has $2k+1$ crossings.
Then $mn>0$, $|m+n|=k + 1$, and $U=(\frac{-1}{2m-1}+\frac{-1}{2n-1})$ 
or $(\frac{-1}{2n-1}+\frac{-1}{2m-1})$ $($see Figure \ref{2kto2k+1}$)$.
\end{thm}

\begin{proof}
First we determine the crossing number of the knot 
$N(\frac{4mn-1}{2m})$. $N(\frac{4mn-1}{2m}) = N(2n + \frac{1}{-2m})$. Thus if $mn<0$, the crossing 
number of the knot $N(\frac{4mn-1}{2m})$ is $2|m+n|$ since a reduced alternating diagram gives 
the crossing number of the corresponding knot. Hence $mn > 0$.  Note that 
$N(\frac{-1}{2m-1} + \frac{-1}{2n-1} + (-1)) = 
N(\frac{-1}{2m-1} + \frac{-2n}{2n-1}) = N(\frac{-(2n-1) - 2n(2m-1)}{-2m}) = N(\frac{4mn-1}{2m})$. 
 Also 
$N(\frac{-1}{2m+1} + \frac{-1}{2n+1} + 1) = N(\frac{-1}{2m+1} + \frac{2n}{2n+1}) = N(\frac{-(2n+1) + 2n(2m+1)}{2m}) = N(\frac{4mn-1}{2m})$. 
Thus the knot $N(\frac{4mn-1}{2m})$ is ambient isotopic to the pretzel knots, $P(2m-1, 2n-1, 1)$ and $P(2m+1, 2n+1, -1)$.  
Thus the crossing number of a knot $N(\frac{4mn-1}{2m})$ is $2|m+n|-1$ if $mn>0$. Hence the crossing number of 
$N(\frac{4mn-1}{2m})$ is $2k+1$ if and only if $mn>0$ and $|m+n|=k + 1$, see Figure \ref{2k+1crossing}.

Applying Theorem \ref{Thm:rts} for $N((U+(-1))+0)=N(U+(-1))=N(\frac{4mn-1}{2m})$ and $N((U+(-1))+1)=N(U+0)=N(2k)$ together with an assumption $mn>0, |m+n|=k+1\ge 4, w=1$, 
we obtain from case (4) $k=m+n-1$, that
$(U+(-1))=\begin{cases}
(\frac{-1}{2m-1}+\frac{-1}{2n-1})\circ(-1)=(\frac{-1}{2m-1}+\frac{-1}{2n-1})+(-1)&
 \mbox{or}\\
(\frac{-1}{2n-1}+\frac{-1}{2m-1})\circ(-1)=(\frac{-1}{2n-1}+\frac{-1}{2m-1})+(-1)&
\end{cases}$, 
and so 
$U=
\begin{cases}
(\frac{-1}{2m-1}+\frac{-1}{2n-1})& \mbox{or}\\
(\frac{-1}{2n-1}+\frac{-1}{2m-1})&
\end{cases}$, 
see Figure \ref{2kto2k+1}.  When $mn>0$ and $|m+n|=k + 1$, case (3) does not occur and cases (1) and (2) are included in case (4).
\end{proof}

\begin{figure}[htbp]
\begin{center}
\includegraphics[scale=0.6]{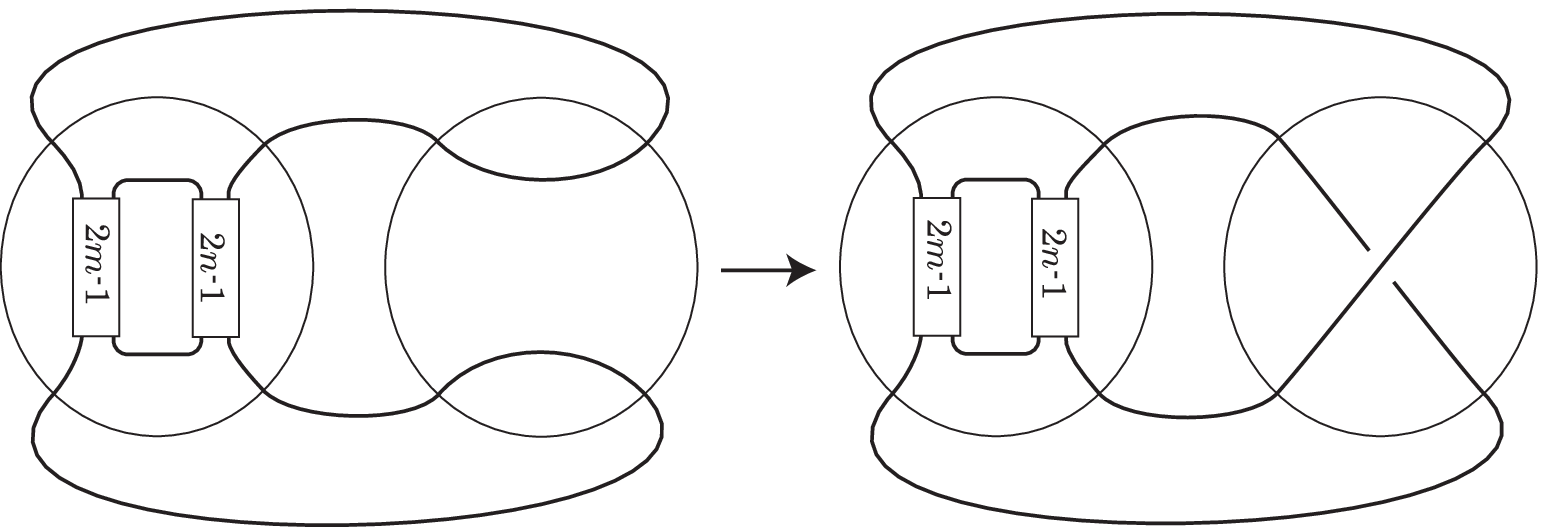}

\hspace{13mm}$N(2k)$\hspace{28mm}$P(2m-1,2n-1,1)$
\caption{Xer recombination from $N(2k)$ to $P(2m-1,2n-1,1)$}
\label{2kto2k+1}
\end{center}
\end{figure}

By Theorem \ref{t-equival}, a $(0,-1)$-move is equivalent to a $(0,\frac{1}{w})$-move,
because $\frac{1}{w} = \frac{1}{-1 - h}$ for an integer $h$ ($h = -w - 1$). Then we obtain Corollary \ref{Cor:Xer2k} below from Theorem \ref{Thm:Xer2k}.

\begin{cor}\label{Cor:Xer2k}
Suppose $N(U+0)= N({2k}) = T_{2,2k}$ and $N(U+\frac{1}{w})= N(\frac{4mn-1}{2m})$, where $k\ge 1$ and $N(\frac{4mn-1}{2m})$ has $2k+1$ crossings.
Then $mn>0$, $|m+n|=k+1$, and $U=(\frac{-1}{2m-1}+\frac{-1}{2n-1}) \circ(-w-1, 0)$ or $(\frac{-1}{2n-1}+\frac{-1}{2m-1}) \circ(-w-1, 0)$. 
\end{cor}


In the following corollaries, we apply Theorem \ref{Thm:Xer2k} to the cases involving
$(2,2k)$-torus links with linking number $-k$ where
$k=3,4,5$,



\begin{cor}\label{6cat}
Suppose $N(U+0)=N(6)$.
\begin{enumerate}
\item If $N(U+(-1))=7_2$ $(=N(\pm\frac{11}{2}))$, then $U=(-\frac{6}{5})$, and so $N(U+(-1))=N(\frac{11}{2})$.
\item If $N(U+(-1))=7_4$ $(=N(\pm\frac{15}{4}))$, then $U=(\frac{-1}{3}+\frac{-1}{3})$, and so $N(U+(-1))=N(\frac{15}{4})$.
\end{enumerate}
\end{cor}


\begin{cor}\label{8cat}
Suppose $N(U+0)=N(8)$.
\begin{enumerate}
\item If $N(U+(-1))=9_2$ $(=N(\pm\frac{15}{2}))$, then $U=(-\frac{8}{7})$, and so $N(U+(-1))=N(\frac{15}{2})$,.
\item If $N(U+(-1))=9_5$ $(=N(\pm\frac{23}{4}))$, then $U=(\frac{-1}{3}+\frac{-1}{5})$ or $U=(\frac{-1}{5}+\frac{-1}{3})$ 
and so $N(U+(-1))=N(\frac{23}{4})$.
\end{enumerate}
\end{cor}


\begin{cor}\label{10cat}
Suppose $N(U+0)=N(10)$.
\begin{enumerate}
\item If $N(U+(-1))=11a247$ $(=N(\pm\frac{19}{2}))$, then $U=(-\frac{10}{9})$, and so $N(U+(-1))=N(\frac{19}{2})$.
\item If $N(U+(-1))=11a343$ $(=N(\pm\frac{31}{4}))$, then $U=(\frac{-1}{3}+\frac{-1}{7})$ or $U=(\frac{-1}{7}+\frac{-1}{3})$, and so $N(U+(-1))=N(\frac{31}{4})$).
\item If $N(U+(-1))=11a363$ $(=N(\pm\frac{35}{6}))$, then $U=(\frac{-1}{5}+\frac{-1}{5})$, and so $N(U+(-1))=N(\frac{35}{6})$.
\end{enumerate}
\end{cor}

%

\subsection{Xer recombination from trefoil knot to Hopf link}
Our original description of Xer recombination was a simple one.  So far we have only described one type of Xer action.  The action of proteins depends on reaction conditions.  
For Xer, its action will be affected by both accessory proteins and the sequence to which it binds.  Xer is a site-specific recombinase (meaning it acts on specific DNA sequences).  But there are several different specific sequences upon which it can act.  Previously we described  the experiment in which $(2, 2k)$  torus links were converted into knots with $2k+1$ crossings.   This case involved the proteins XerC, XerD, and PepA acting on the DNA sequence {\it psi}.  In this case in order for Xer to act, these proteins trapped three DNA crossings (Figure  \ref{xer4cat}).

We will now describe a second set-up for  Xer recombination.  The proteins  XerC, XerD, and FtsK  can also act on the E. coli {\it dif} sequence.   
To distinguish this second type of action we will refer to this as XerCD-{\it dif}-FtsK recombination.  This system can resolve dimers.    It can also unlink DNA links in vitro  (i.e., in a test tube) \cite{Ip} and in vivo (i.e., in the cell) \cite{embo}.  
For  XerCD-{\it dif}-FtsK recombination, the topology of the protein-bound DNA is much simpler.  It is believed that FtsK is responsible for setting up a much simpler protein-bound DNA topology in which there is a projection where no DNA crossings are trapped.  For a movie of a proposed model, see supplementary data in  \cite{Ip}.

XerCD-{\it dif}-FtsK can unlink
the $(2,2k)$-torus link with linking number $\pm k$.  When the linking number is $-k$, the observed products were the unknot and the unlink of two componts and the proposed pathway is believed to be $(2,2k)$-torus link $\rightarrow$ unknot $\rightarrow$ unlink \cite{Ip}.  The related tangle equations are easily solved using \cite{HS, Darcy15072006}.  
When the linking number is $k$,  
a {stepwise unlinking model} was proposed 
\cite{Ip} in which each round of recombination reduces the  crossing number by one:  
the pathway is the $(2,2k)$-torus link, the $(2,2k-1)$-torus knot,
the $(2,2k-2)$-torus link, $\cdots$, the $(2,4)$-torus link, the trefoil
knot, the Hopf link, the trivial knot, the trivial link.
See \cite{SIV} for a mathematical study.
If we assume each recombination is modeled by a coherent band surgery,
the following corollary of Theorem 3.1 characterizes the band surgery between the trefoil knot and
the Hopf link.
The band surgery between the Hopf link and the trivial knot is
characterized by \cite{Thompson1989} and \cite{HS} and the band surgery between
the trivial knot and the trivial link is characterized by \cite{Sc}.

 \begin{cor}\label{Thm:Hopftrefoil}
 If $N(U+0)= N(3)$ $($trefoil knot$)$ and $N(U+\frac{1}{w})= N(2)$ $($Hopf link$)$,
  then $U=(\frac{3}{-3w-2})$.
 If $N(U+0)= N(3)$ and $N(U+(-1))= N(2)$,
  then $U=(3)$ $($see Figure \ref{Ht}$)$.
 \end{cor}

 \begin{proof}
 We apply Theorem \ref{Thm:rts} by putting $m=n=k=w=-1$. 
 Then we obtained the solutions (1),(2) and (3). Each of them shows $U=(3)$.
 \end{proof}

\begin{figure}[htbp]
\begin{center}
\includegraphics[scale=0.6]{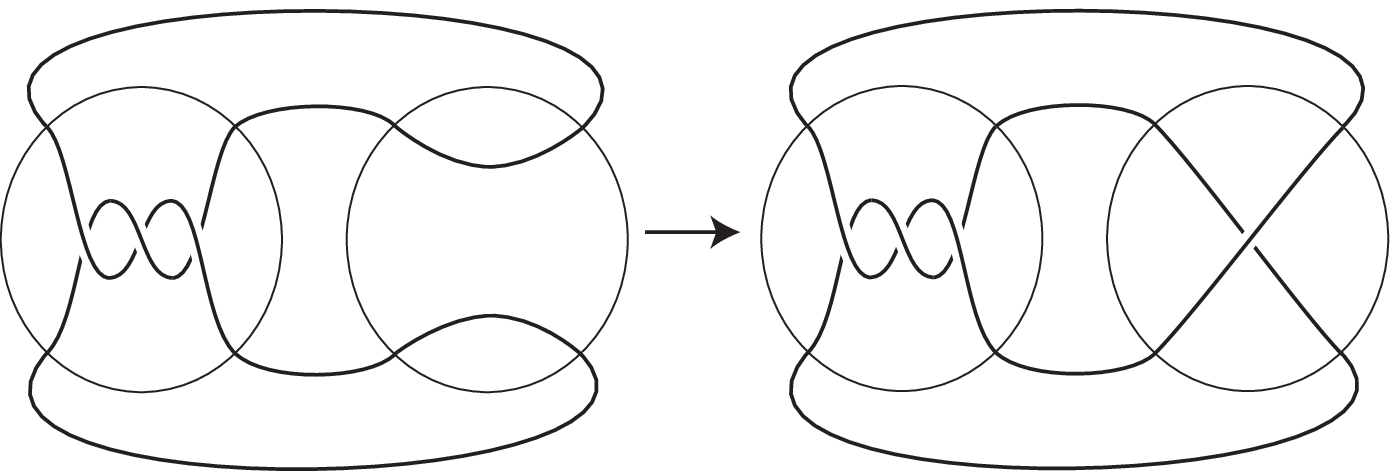}

trefoil knot\hspace{28mm} Hopf link
\caption{Xer recombination from the trefoil knot to the Hopf link.}
\label{Ht}
\end{center}
\end{figure}

\subsection{A comment on converting $N(4)$ into the trefoil knot}\label{Sec:4cat}
In the remark below, we look at one solution for converting $N(4)$ into $N(3)$.

\begin{rem}\label{Lem:4cat}
 If $U=(4)$, then $N(U+0)= N(4)$  where $N(4)$ has linking number +2 and $N(U+(-1))= N(3)$.
   \end{rem}
 
If $N(4)$ has linking number $+2$, we cannot apply Theorem \ref{Thm:rts}.
In this case the coherent band corresponding to the rational tangle surgery
can be isotoped to lie in the genus one Seifert surface of $N(4)$, but
not in the minimal Seifert surface of $N(3)$
(See Theorem \ref{hs}).


\section{Proof of Theorem \ref{Thm:rts}}\label{proof}

In this section we prove Theorem \ref{Thm:rts} when  $m, n \not= 0$.  Thus we wish to solve the system of equations $N(U+0)=N(\frac{4mn-1}{2m})$ and $N(U+\frac{1}{w})= (2, 2k)$-torus link.  
We will use that this rational tangle surgery corresponds to a band surgery. We assume the band surgery is coherent.
%
For a coherent band surgery, the relation between the signatures of $L$ and $L_b$ is known.

\begin{thm}{\rm (\cite[Lemma 7.1]{Mu2})}\label{mu2}
If $L_b$ is obtained from $L$ by a coherent band surgery,
then $|\sigma (L)-\sigma (L_b)|\le 1$.
\end{thm}

We consider the case where $L$ is a $2$-bridge knot $N(\frac{4mn-1}{2m})$
and $L_b$ is $N(2k)$. 
Suppose $N(2k)$ has linking number $k$ where $|k| > 2$. 
Then the absolute value of its signature is $2|k|-1$ 
while $|\sigma (N(\frac{4mn-1}{2m}))|=0$ or $2$.
Thus when $|k|>2$, the system of equations, $N(U+0)=N(\frac{4mn-1}{2m})$ and $N(U+\frac{1}{w})= N(2k)$ where $N(2k)$ has linking number $k$ has no solution when the rational tangle surgery corresponds to a 
coherent band surgery.
Since $N(2) = N(-2)$ as unoriented links, we can assume that the Hopf link, $N(2k)$ has linking number $-k$ where $k = \pm 1$.  Hence from here on, we will assume $N(2k)$ has linking number $-k$.


\begin{figure}[htbp]
\begin{center}
\includegraphics[scale=0.6]{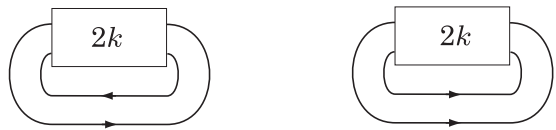}

\end{center}
\caption{
$(2,2k)$-torus links:  the left diagram has linking number $-k$ while the right diagram has linking number $k$.}
\label{ori}
\end{figure}

\begin{thm}[\cite{HS}]\label{hs}
Suppose that $b$ is a band of a coherent band surgery from $L$ to $L_b$. 
Then $\chi(L)\le \chi(L_b)-1$ if and only if $L$ has a minimal
genus Seifert surface containing $b$.
\end{thm}

In our case $\chi(L)=-1$ and $\chi(L_b)=0$ or $2$.
The link $L$ is a genus one $2$-bridge knot and genus one Seifert surfaces of $L$
are characterized in \cite{HT}; 
there are only two Seifert surfaces of $L$ up to equivalence as shown in Figure \ref{ss}.
Here two Seifert surfaces for an oriented link are said to be equivalent if there exists an ambient isotopy 
such that one Seifert surface is moved to the other by the isotopy
and the link is fixed as a set throughout the isotopy.
The surfaces $S_1$ and $S_2$ in Figure \ref{ss} are genus one Seifert surfaces of 
$L$ which are obtained by plumbing of two annuli with $m$ and $n$ full twists. 
Moreover, by \cite{K} it is known that $S_1$ is equivalent to $S_2$ if and only if $m=\pm 1$ or $n=\pm 1$.
Let $S$ be a minimal Seifert surface of $L$ which contains $b$ (i. e., $S = S_1$ or $S_2$).
The Seifert surface $S_1$ can be isotoped to $S_2$ via a $\pi$ rotation which moves $L$, but we will prove results for the stronger definition of equivalence given above.

\begin{figure}[htbp]
\begin{center}
\includegraphics[scale=0.6]{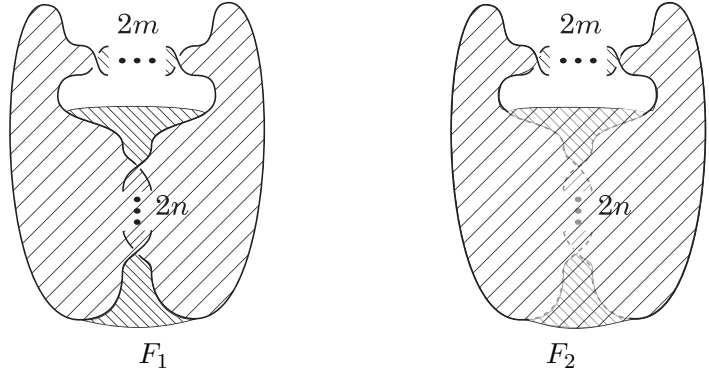}

$F_1$\hspace{5cm}$F_2$
\end{center}
\caption{Minimal genus Seifert surfaces for $N(\frac{4mn-1}{2m})$.}
\label{ss}
\end{figure}


We will give a parametrization of bands attached to $L$
which are contained in $S$.
If $b$ is contained in $S$ and if $S - b$ is connected, then Cl($S-b$) is an annulus in $S^3$.
Let $\gamma_b$ denote a core of the annulus.  
 Note that the boundary of this annulus is exactly $L_b$, so $L_b=N(2k)$ if and only if $\gamma_b$ is a trivial knot. 
A band $b$ determines a unique knot $\gamma_b$ (up to isotopy in $S$).  Conversely, given a knot in $S$, there exists a unique band in $S$ (up to isotopy in $S$) which is disjoint from that knot.
Thus there exists a one to one correspondence between ambient isotopy classes of  bands $b$ in $S$ and ambient isotopy classes of unoriented knots $\gamma_b$ in $S$.
In order to parametrize a band $b$ in $S$,
we will parametrize such 
a knot $\gamma_b$ in $S$.  
Recall that $S$ is a plumbing of two annuli with $m$ and $n$ full twists. 
Let $c_M$ (resp. $c_N$) be the oriented unknot shown in Figure \ref{para} which 
spans a disk that transversely meets 
the annulus with $m$ (resp. $n$) full twists in an arc. 
For parametrization, we give an orientation of $\gamma_b$.
Let $p=lk(\gamma_b,c_M)$ and $q=lk(\gamma_b,c_N)$. 
Then $\gamma_b$ is parametrized by $\pm (p,q)$.

\begin{figure}[htbp]
\begin{center}
\includegraphics[scale=0.6]{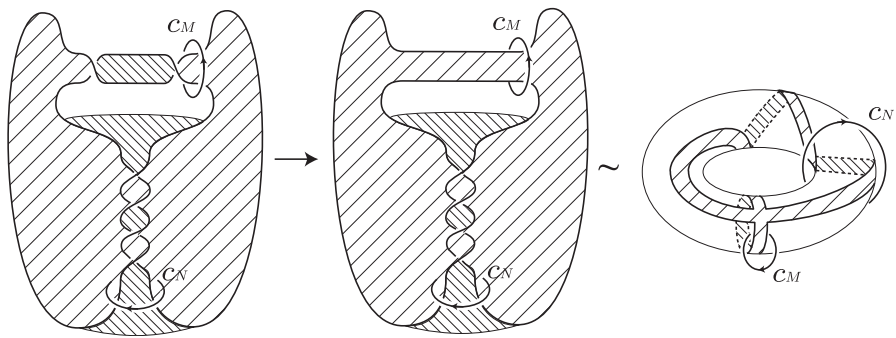}
\end{center}
\caption{$(-m)$-twist along $c_M$.}
\label{para}
\end{figure}

We consider a $(-m)$-twist along $c_M$. 
Then after the twist, $S$ lies on a standard torus in $S^3$ as shown in Figure \ref{para}.
Hence $\gamma_b$ becomes a torus knot after the twist.
Motegi \cite{Mo} characterized twists on an unknot which yield torus knots.

\begin{thm}{\rm (\cite[Theorem 3.8.]{Mo})}\label{mo}
Suppose a knot $K_{-m}$, which is obtained from a trivial knot $K$ by a $(-m)$-twist along a trivial knot $C$, is a torus knot.
Then, except for trivial examples, $m=\pm 1$.
\end{thm}
Trivial examples are shown in Figure \ref{tri}.
In trivial examples, $C$ is ambient isotopic (in $S^3-K$) to the core loop
of a solid torus whose boundary torus contains $K$ and $K_{-m}$.

\begin{figure}[htbp]
\begin{center}
\includegraphics[scale=0.6]{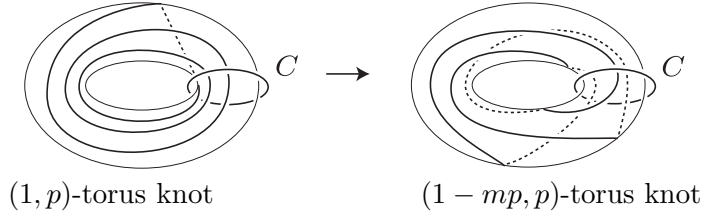}

$(1,p)$-torus knot
\hspace{2.5cm}
$(1-mp,p)$-torus knot
\end{center}
\caption{Trivial example: A $(1-mp,p)$-torus knot is obtained from a $(1,p)$-torus knot, which is a trivial knot, by $(-m)$-twist along a trivial knot $C$.}
\label{tri}
\end{figure}

By applying Theorem \ref{mo} for $C=c_M$ and $K=\gamma_b$,
we obtain the following lemma.

\begin{lemma}\label{lem0}
If $\gamma_b$ is a trivial knot, then one of the following holds.
\begin{enumerate}
\item [(1)] $|p|\le 1$ and $|q|\le 1$.
Namely $(p,q)=\pm (1,0)$, $\pm (0,1)$, $\pm (1,1)$, or $\pm (1,-1)$.
\item [(2)] $n=1$ and $(p,q)=\pm (1,-2)$.
\item [(3)] $n=-1$ and $(p,q)=\pm (1,2)$.
\item [(4)] $m=\pm 1$.
\end{enumerate}
\end{lemma}

\begin{rem}
If $(1)$ of Lemma \ref{lem0} holds, then for any $m$ and $n$, $\gamma_b$ is trivial.
If $(2)$ or $(3)$ holds, then for any $m$, $\gamma_b$ is trivial.
\end{rem}
\begin{proof}[Proof of Lemma \ref{lem0}]
Let $\gamma_b'$ be a knot which is obtained from $\gamma_b$ by $(-m)$-twist along $c_M$.
The standard torus on which $\gamma_b'$ lies divides $S^3$ into the union of two solid tori.  The loop $c_N$ maps to the core knot for one of these tori.  Let $c$ be the core knot for the other torus as in Figure \ref{core}.

\begin{figure}[htbp]
\begin{center}
\includegraphics[scale=0.6]{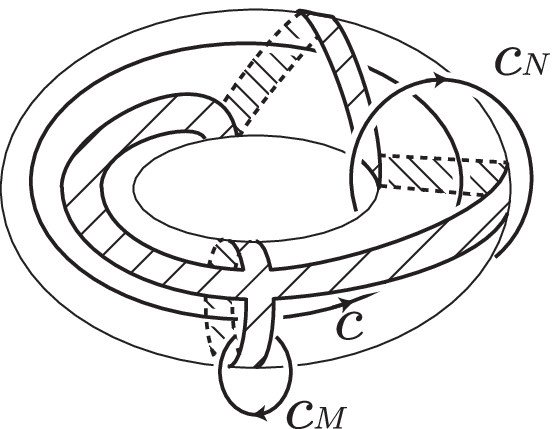}
\caption{
}
\label{core}
\end{center}
\end{figure}

Then linking numbers can be calculated as follows.
\[lk(\gamma_b',c_M)=lk(\gamma_b,c_M)=p,\]
\[lk(\gamma_b',c_N)=lk(\gamma_b,c_N)=q,\]
\[lk(\gamma_b',c)=p+nq.\]

By applying Theorem \ref{mo} for $C=c_M, K=\gamma_b, K_{-m}=\gamma_b'$, 
except trivial examples, we have the conclusion (4) 
of Lemma \ref{lem0}.
Thus it is enough to consider the trivial examples.
Since a $(-m)$-twist along $C_M$ represents a trivial example, we have 
\begin{equation}
\gamma_b'=(\pm 1-mp,p)\mbox{-torus knot}.\label{ptorus}
\end{equation}
In other hand, from the linking numbers of $\gamma_b'$ with $c_N$ and with $c$ above, we have
\begin{equation}
\gamma_b'=(q,p+nq)\mbox{-torus knot}.\label{qtorus}
\end{equation}

Suppose $\gamma_b'$ is the trivial knot. From equation (\ref{ptorus}), $|p|=1$ or $|\pm 1-mp|=1$.  
From  equation (\ref{qtorus}), $|q|=1$ or $|p+nq|=1$.
If $|p|=|q|=1$, then the conclusion (1) of Lemma \ref{lem0} holds.
If $|p|=|p+nq|=1$, then $|nq|=1\pm|p|=0$ or $2$, so one of the conclusions (1),(2),(3) of Lemma \ref{lem0} holds.
If $|p|\neq 1$ and $|\pm 1-mp|=1$, then $|mp|\le 2$, so $(p,q)=\pm(0,1)$ or $m=\pm 1$, i.e. the conclusion (1) or (4) of Lemma \ref{lem0} holds.

Suppose $\gamma_b'$ is a non-trivial torus knot. 
From equation (\ref{ptorus}), $|p|\ge 2$.  
From equation (\ref{qtorus}), $|q|\ge 2$.  The
integers $p$ and $q$ are relatively prime, so $|p|\neq|q|$. 
By contrasting  equations (\ref{ptorus}) and (\ref{qtorus}), $(\pm 1-mp,p)=\pm(q,p+nq)$.
Since $nq\neq 0$,
\begin{equation}
\pm 1-mp=-q, \mbox{ and}\label{1-mp=-q}
\end{equation}
\begin{equation}
p=-(p+nq).\label{-p=p-nq}
\end{equation} 
If $|q|<|p|$, by  equation (\ref{1-mp=-q}),
$|q|<|p|\le |mp|=|q\pm 1|\le |q|+1$. then $|p|=|mp|$, 
so the conclusion  (4) of Lemma \ref{lem0} holds.
If $|q|>|p|$,  by  equation (\ref{-p=p-nq}), $2|p|=|nq|>|np|$, then $|n|=1$ and $2|p|=|q|$.
Since integers $p$ and $q$ are relatively prime and $|p|\ge 2$, it does not happen.
\end{proof}

Next we state a well known lemma of a braid presentation for 
a trivial knot (see, for example, \cite{MR2191944}).

\begin{lemma}\label{lem1} 
If a trivial knot has a positive or negative $n$-braid presentation with $m$ crossings, 
then $m=n-1$.
\end{lemma}

We say that two bands attached to a link $L $ are {\it equivalent with respect to} $L$ 
if there exists an ambient isotopy of $S$ in $S^3$
such that one band is moved to the other by the isotopy
and $L$ is fixed as a set throughout the isotopy.
Now we state the main theorem of this section.

\begin{thm}\label{Thm:fix}
Let $L$ be a $2$-bridge knot $N(\frac{4mn-1}{2m})$ in $S^3$
with $m, n\neq 0$.
Suppose that $b$ is a band of a coherent band surgery from $L$ to $L_b$, and $L_b$ is a $2$-bridge link $N(2k)$
with 
linking number $-k$.
Then the band $b$ is equivalent to one of the six bands $b_1,b_2,b_3,b_4,b_5$ and $ b_6$ in Figure \ref{fixb} with respect to $L$.
\end{thm}

\begin{figure}[htbp]
\begin{center}
  \includegraphics[scale=0.71]{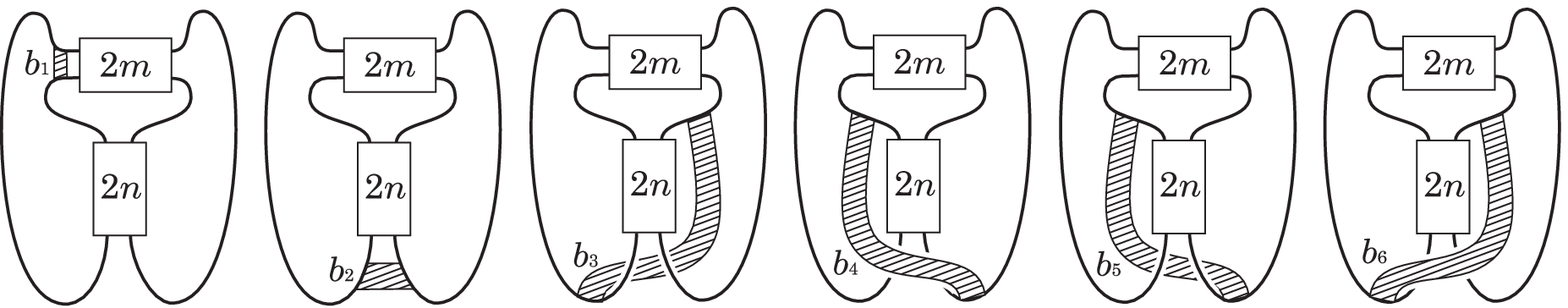}
    
    \hspace{6mm}
    $k=m$\hspace{12.5mm}
    $k=n$\hspace{6.5mm}
    $k=m+n+1$\hspace{2mm}
    $k=m+n+1$\hspace{2mm}
    $k=m+n-1$\hspace{2mm}
    $k=m+n-1$
\end{center}
\caption{
}
\label{fixb}
\end{figure}
\begin{rem}\label{rem}
\begin{enumerate}
\item 
The band for $(p,q)=\pm (1,0)$ and $S=S_1$ or $S_2$ corresponds to $b_1$.
\item 
The band for $(p,q)=\pm (0,1)$ and $S=S_1$ or $S_2$ corresponds to $b_2$.
\item 
The band for $(p,q)=\pm (1,1)$ and $S=S_1$ $($resp. $S_2)$ corresponds to $b_3$ $($resp. $b_4)$.
\item 
The band for $(p,q)=\pm (1,-1)$ and $S=S_1$ $($resp. $S_2)$ corresponds to $b_5$ $($resp. $b_6)$.
\end{enumerate}
\end{rem}

\begin{rem}
Suppose that $m=\pm1$ or $n=\pm1$. Then two Seifert surfaces $S_1$ and $S_2$ are equivalent (fixing $L$). By the ambient isotopy of this equivalence, we have the following:
If $m=1$ $($resp. $m=-1)$, then the three bands $b_2,b_5,b_6$ $($resp. $b_2,b_3,b_4)$ are equivalent to each other with respect to $L$. 
If $n=1$ $($resp. $n=-1)$, then the three bands $b_1,b_5,b_6$ $($resp. $b_1,b_3,b_4)$ are equivalent to each other with respect to $L$. 
If $m=n=1$ $($resp. $m=n=-1)$ $($i.e. $L$ is the trefoil knot$)$, 
then the four bands $b_1,b_2,b_5,b_6$ $($resp. $b_1,b_2,b_3,b_4)$ are equivalent
to each other with respect to $L$, and the two bands $b_3,b_4$ $($resp. $b_5,b_6)$ are equivalent with respect to $L$. If $m=1,n=-1$ $($resp. $m=-1,n=1)$ $($i.e. $L$ is the figure eight knot$)$, 
then the three bands $b_1,b_3,b_4$ $($resp. $b_1,b_5,b_6)$ are equivalent
to each other with respect to $L$, and the three bands $b_2,b_5,b_6$ $($resp. $b_2,b_3,b_4)$ are equivalent to each other with respect to $L$. 

\end{rem}

\begin{proof}[Proof of Theorem \ref{Thm:fix}]
Let $L$ be a $2$-bridge knot $N(\frac{4mn-1}{2m})$ in $S^3$
with $m, n\neq 0$. 
Since $L$ is symmetric for $m$ and $n$, i.e. $N(\frac{4mn-1}{2n})$ 
is ambient isotopic to $N(\frac{4mn-1}{2m})$, 
without loss of generality, we may assume that $|m|\ge |n|\ge 1$.
Suppose, for a band $b$, $L_b$ is a $2$-bridge link $N(2k)$
with linking number $-k$.
We may assume that $b\subset S$ where $S$ is a minimal Seifert surface
of $L$, and so $S=S_1$ or $S_2$ in Figure \ref{ss}.

Suppose the conclusion (1) of Lemma \ref{lem0} holds.
Then the band is equivalent to one of $b_1,b_2,b_3,b_4,b_5$ and $b_6$ 
with respect to $L$, see Remark \ref{rem}.

Suppose the conclusion (2) of Lemma \ref{lem0} holds. 
In this case $S_1$ and $S_2$ are equivalent. 
By the ambient isotopy of this equivalence, the band for 
$(p,q)=\pm(1,-2)$ and $S=S_1$ (resp. $S=S_2$) is equivalent to $b_4$ (resp. $b_3$) with respect to $L$.

Suppose the conclusion (3) of Lemma  \ref{lem0} holds. 
Similarly to above, the band for 
$(p,q)=\pm(1,2)$ and $S=S_1$ (resp. $S=S_2$) is equivalent to $b_6$ (resp. $b_5$) with respect to $L$.

\begin{figure}[htbp]
\begin{center}
 \includegraphics[scale=0.6]{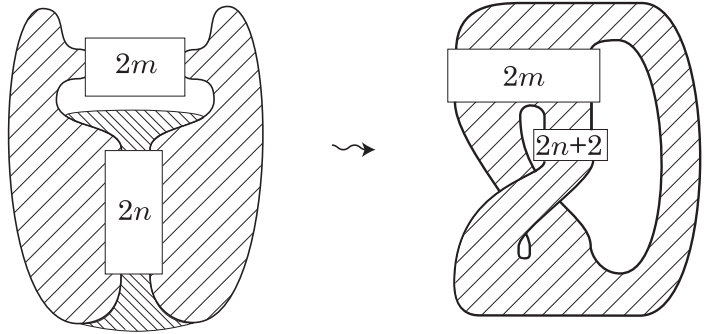}
\end{center}
\caption{Recall that a positive (negative) number inside a box corresponds to right (left) handed half twists.  For the figure on the left, the box containing $2m$ represents $m$ horizontal twists. The remaining three boxes contain vertical twists. 
}
\label{bp}
\end{figure}

\begin{figure}[htbp]
  \begin{center}
      \includegraphics[scale=0.6]{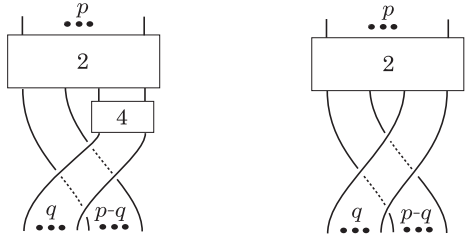}
  \end{center}
  \caption{On the left : $m=n=1$. On the right : $m=1,n=-1$.}
  \label{m=n=1}
\end{figure}


\begin{figure}[htbp]
\begin{center}
      \includegraphics[scale=0.6]{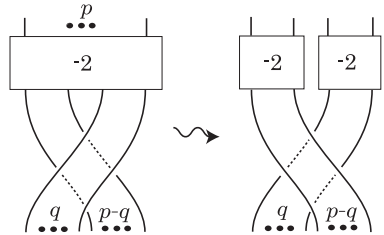}
      \end{center}
\caption{$m=n=-1,p>q>0$:  The $p$-string braid on the left simplifies via two right-handed half-twists to the negative braid on the right.  Note that the right braid has $p - 1$ crossings (and hence its closure is the unknot) if and only if $p = 2$ and $q = 1$.   }
\label{m=n=-1}
\end{figure}

\begin{figure}[htbp]
\begin{center}
      \includegraphics[scale=0.6]{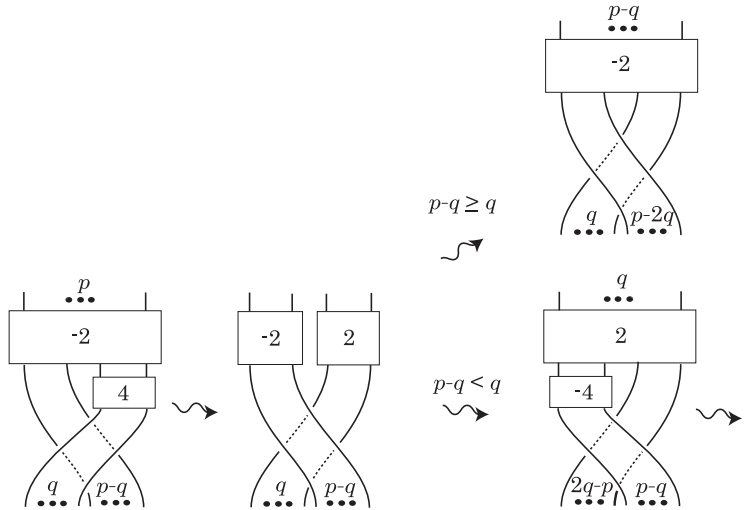}
      \end{center}
\caption{$m=-1,n=1,p>q>0$:  The braid on the left simplifies to one of the braids on the right.  If $p - q \geq q$, then a $(p-q)$-string negative braid is obtained (top right).
This $(p-q)$-string braid closes to the unknot if and only if $p - q = q = 1$
If  $p - q < q$, then a $q$-string  braid is obtained which can be further simplified (bottom right).  This $q$-string braid closes to the unknot if and only if $(p, q) = (F_{i+1}, F_i)$  where $F_i$ is the $i$-th Fibonacci number.  
}
\label{m=-1n=1}
\end{figure}
From now on we will consider the remaining case where $|m|=|n|=1$.
There are symmetries of $\gamma_b$: i.e. the parameters $(m,n,p,q)=(a,b,c,d),(b,a,d,c)$ and $(-a,-b,c,-d)$ determine the same knot type of $\gamma_b$ up to mirror image.
Thus it is enough to consider the case where $p>q>0$. 
Then we obtain a $p$-string braid presentation for $\gamma_b$ by moving $S$ as shown in Figure \ref{bp}.
Using this braid presentations for $\gamma _b$ and Lemma \ref{lem1}, 
we will decide whether or not the knot $\gamma _b$ is trivial.
There are three cases for $m$ and $n$: 
the first is where $m=1$ and $n=\pm1$; 
the second is where $m=-1$ and $n=1$; 
and the third is where $m=n=-1$.
%
In the first case,
$\gamma_b$ has a positive braid presentation as shown in Figure \ref{m=n=1}, and hence  $\gamma_b$ is a non-trivial knot.
In the second case,
by simplifying a braid presentation, 
$\gamma_b$ has a negative braid presentation as shown in Figure \ref{m=n=-1}.
Hence $\gamma_b$ is a non-trivial knot
except when $(p,q)=(2,1)$.
In the third case,
we can simplify a braid presentation inductively if necessary as shown in Figure \ref{m=-1n=1}.
Then $\gamma_b$ is a trivial knot 
if and only if $(p,q)=(F_{i+1},F_i)$, where $i$ is any positive integer 
and $F_i$ is the $i$-th Fibonacci number; 
$F_0=0,F_1=1,F_2=1,F_3=2,F_4=3,F_5=5,\ldots$.
By reconsidering the symmetries of $\gamma_b$, 
we can summarise a necessary and sufficient condition 
for a pair $(p,q)$ to determine a trivial knot $\gamma_{b}$ 
for $m=n=1$ and for $m=1,n=-1$ as follows.

Suppose that $m=n=1$, i.e. $L$ is the trefoil knot.
A knot $\gamma _b$ is trivial if and only if $(p,q)=\pm (0,1)$,$\pm (1,0)$, $\pm (1,1)$,
$\pm (1,-1)$, $\pm (2,-1)$, or $\pm (1,-2)$. 
Now $L$ is a fibered knot and $S$ is a fibered surface.
Let $\eta :S\to S$ be the monodromy map and let $\eta_{*}:H_1(S)\to H_1(S)$ be the homomorphism induced by $\eta$. 
We regard $(p,q)$ as an element of $H_1(S)$ and define two elements of $H_1(S)$ to be equivalent 
if $\eta_{*}^{k}$ maps one to the other for an integer $k$.
Then $\pm (1,0),\pm (0,1)\pm (1,-1)$ belong to an equivalence class, 
and $\pm (1,1),\pm (2,-1),\pm (1,-2)$ belong to another one. 
This implies that the band is equivalent to $b_1$ or $b_3$ 
with respect to $L$, 
see Remark \ref{rem}.

Suppose that $m=1$ and $n=-1$, i.e. $L$ is the figure eight knot.
A knot $\gamma_b$ is trivial if and only if 
$(p,q)=\pm (F_i,F_{i+1})$ or $\pm (F_{i+1},-F_i)$ 
for any non-negative integer $i$. 
Now $L$ is a fibered knot and $S$ is a fibered surface. 
Let $\eta :S\to S$ be the monodromy map and 
let $\eta^{*}:H_1(S)\to H_1(S)$ be the homomorphism induced by $\eta$. 
We regard $(p,q)$ as an element of $H_1(S)$ and define two elements of $H_1(S)$ to be equivalent 
if $\eta_{*}^{k}$ maps one to the other for an integer $k$.
Then a set 
$A=\{ (F_{2i-1},-F_{2i-2})\ |\ i\in\mathbb{N}\}\cup
\{(F_{2i-1},F_{2i})\ |\ i\in\mathbb{N}\}$ is an equivalence class 
and a set $B=\{ (-F_{2i},F_{2i-1})\ |\ i\in\mathbb{N}\}\cup
\{(F_{2i-2},F_{2i-1})\ |\ i\in \mathbb{N}\}$ is another one. 
Note that the set $A$ contains $(1,0)$ and the set $B$ contains $(0,1)$. 
This implies that the band is equivalent to $b_1$ or $b_2$ 
with respect to $L$, see Remark \ref{rem}.
This completes the proof of Theorem \ref{Thm:fix}.
\end{proof}

In Figure \ref{fixb}, the two unions $L\cup b_3$ and $ L\cup b_4$ 
(resp. $L\cup b_5$ and $L\cup b_6$) of a knot $L$ and bands are ambient isotopic to each other, and they are ambient isotopic to the third (resp. the fourth) of Figure \ref{nfixb}.
Then we obtain the following theorem.

\begin{figure}[htbp]
\begin{center}
\includegraphics[scale=0.8]{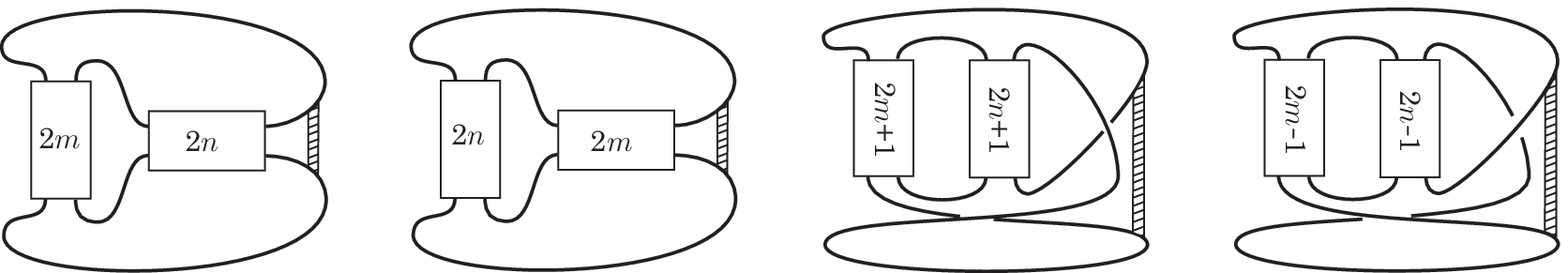}

\hspace{5mm}$k=m$\hspace{2.5cm}$k=n$\hspace{2cm}$k=m+n+1$\hspace{1cm} $k=m+n-1$
\end{center}
\caption{
}
\label{nfixb}
\end{figure}

\begin{thm}\label{Thm:nfix}
Let $L$ be a $2$-bridge knot $N(\frac{4mn-1}{2m})$ in $S^3$ with $m, n\neq 0$.
Suppose that $b$ is a band of a coherent band surgery from $L$ to $L_b$, and $L_b$ is a $2$-bridge link $N(2k)$
with 
linking number $-k$.
Then the union $L\cup b$ of a knot $L$ and a band $b$ is ambient isotopic to one of
four in Figure \ref{nfixb}.
\end{thm}

\begin{proof}[Proof of Theorem \ref{Thm:rts}]
We consider a 
$(0,\frac{1}{w})$ move
as a band surgery as shown in Figure \ref{br-surgery}. 
Then we obtain Theorem \ref{Thm:rts} from Theorem \ref{Thm:nfix} immediately.
\end{proof}

\section{Non-band rational tangle surgery case}\label{non-band}

Let $c(K)$ be the crossing number of the knot $K$.  
In this section we will characterize all non-band rational tangle surgeries
on $N(2k)$ yielding $N(\frac{4mn-1}{2m})$,
where $c(N(2k))= 2k$ and $c(N(\frac{4mn-1}{2m}))=2k+1$,
for $k=3, 4$ or $5$.

Suppose $N(U+0)=N(2k)$ ($=(2,2k)$-torus link), and $N(U+\frac{t}{w})=$ a $2$-bridge knot $N(\frac{z}{v})$. Per below the solutions to this system of equations are of the form
  $\frac{t}{w}=\frac{z-2kv'}{v'-(z-2kv')h}$ and $U=(\frac{2k}{2kh+1})$, 
where $h$ is any integer and $v'$ is an integer which satisfies $v'\equiv v^{\pm 1}$ mod $z$, and so $N(\frac{z}{v'})=N(\frac{z}{v})$.
We show that there are no  other solutions of the non-band rational tangle surgery for the cases where $\frac{z}{v}=\frac{4mn-1}{2m}$, $mn>0$, $|m+n|=k+1$ 
(which is the condition for $c(N(\frac{4mn-1}{2m}))=2k+1$, and $k=3,4$ or $5$.

\begin{thm}
\label{non-band6cat}
Suppose $N(U+0)=N(6)$ and $t\neq\pm 1$.
\begin{enumerate}
\item If $N(U+\frac{t}{w})=7_2$ $(\frac{z}{v}=\pm\frac{11}{2})$, 
then $\frac{t}{w}=\frac{11-6v'}{v'-(11-6v')h}$ and $U=(\frac{6}{6h+1})$, 
where $h$ is any integer and $v'$ is an integer which satisfies $v'\equiv \pm 2^{\pm 1}$ mod $11$.
\item If $N(U+\frac{t}{w})=7_4$ $(\frac{z}{v}=\pm\frac{15}{4})$, 
then $\frac{t}{w}=\frac{15-6v'}{v'-(15-6v')h}$ and $U=(\frac{6}{6h+1})$, 
where $h$ is any integer and $v'$ is an integer which satisfies $v'\equiv \pm 4^{\pm 1}$ mod $15$.
\end{enumerate}
\end{thm}

\begin{thm}
\label{non-band8cat}
Suppose $N(U+0)=N(8)$ and $t\neq\pm 1$.
\begin{enumerate}
\item If $N(U+\frac{t}{w})=9_2$ $(\frac{z}{v}=\pm\frac{15}{2})$, 
then $\frac{t}{w}=\frac{15-8v'}{v'-(15-8v')h}$ and $U=(\frac{8}{8h+1})$, 
where $h$ is any integer and $v'$ is an integer which satisfies $v'\equiv \pm 2^{\pm 1}$ mod $15$.
\item If $N(U+\frac{t}{w})=9_5$ $(\frac{z}{v}=\pm\frac{23}{4})$, 
then $\frac{t}{w}=\frac{23-8v'}{v'-(23-8v')h}$ and $U=(\frac{8}{8h+1})$, 
where $h$ is any integer and $v'$ is an integer which satisfies $v'\equiv \pm 4^{\pm 1}$ mod $23$.
\end{enumerate}
\end{thm}

\begin{thm}
\label{non-band10cat}
Suppose $N(U+0)=N(10)$ and $t\neq\pm 1$.
\begin{enumerate}
\item If $N(U+\frac{t}{w})=11a247$ $(\frac{z}{v}=\pm\frac{19}{2})$, 
then $\frac{t}{w}=\frac{19-10v'}{v'-(19-10v')h}$ and $U=(\frac{10}{10h+1})$, 
where $h$ is any integer and $v'$ is an integer which satisfies $v'\equiv \pm 2^{\pm 1}$ mod $19$.
\item If $N(U+\frac{t}{w})=11a343$ $(\frac{z}{v}=\pm\frac{31}{4})$, 
then $\frac{t}{w}=\frac{31-10v'}{v'-(31-10v')h}$ and $U=(\frac{10}{10h+1})$, 
where $h$ is any integer and $v'$ is an integer which satisfies $v'\equiv \pm 4^{\pm 1}$ mod $31$.
\item If $N(U+\frac{t}{w})=11a363$ $(\frac{z}{v}=\pm\frac{35}{6})$, 
then $\frac{t}{w}=\frac{35-10v'}{v'-(35-10v')h}$ and $U=(\frac{10}{10h+1})$, 
where $h$ is any integer and $v'$ is an integer which satisfies $v'\equiv \pm 6^{\pm 1}$ mod $35$.
\end{enumerate}
\end{thm}
 
The above results can also be obtained by using the software TopoIce-R
\cite{Darcy15072006} within Knotplot \cite{SchPhD}.  
This software implements the following theorems.

\begin{thm}
\cite{E}\label{E}
If $N(N+0)=N(\frac{a}{b})$ and $N(U+\frac{t}{w})=N(\frac{z}{v})$ and if $|t|>1$, 
then $U$ is a generalized $M$-tangle or equivalently, 
$U$ is obtained from a finite sum of rational tangles by a circle product with a finite sequence of integers.
\end{thm}

\begin{thm}
\cite[Theorem 3]{D_unor}\label{Thm:D_gM}
$N(U+0)=N(\frac{a}{b})$ and $N(U+\frac{t}{w})=N(\frac{z}{v})$ where $U$ is a generalized $M$-tangle 
if and only if the following hold.
\begin{itemize}
\item[(a)] If $w\not\equiv\pm 1$ mod $t$, then there exists an integer, $b'$ such that $b'b^{\pm 1}=1$ mod $a$, 
       and for any integers $x$ and $y$ such that $b'x-ay=1$, 
       \[N(\frac{z}{v})=N(\frac{tb'+wa}{ty+wx}).\]
       In this case, $U=\frac{a}{b'}$ for all $b'$ satisfying the above.
\item[(b)] If $w\equiv\varepsilon=\pm 1$ mod $t$ $(w=ht+\varepsilon)$, 
       then there exists relatively prime integers, $p$ and $q$,  
       where $p$ may be chosen to be positive, such that, 
       \[N(\frac{z}{v})=N(\frac{tp(pb-qa)+\varepsilon a}{tq(pb-qa)+\varepsilon b}).\]
       In this case, the solutions for $U$ are $(\frac{da-jp}{pb-qa}+\frac{j}{p})\circ (h,0)$ 
       and $(\frac{j}{p}+\frac{da-jp}{pb-qa})\circ (h,0)$, for all $p,q$ satisfying the above, 
       $d$ and $j$ are integers such that $pd-qj=1$ 
       (note, the choice of $j$ and $d$ such that $pd-qj=1$ has no effect on $U$).
\end{itemize}
\end{thm}

\begin{cor}
\cite[Corollary 2]{D_unor}\label{Cor:D}
Suppose
 $bx - ay = 1$, $N(U + \frac{0}{1}) = N({\frac{a}{b}})$ and 
$N(U + \frac{t}{w}) = N(\frac{z}{v})$ 
where $N(\frac{a}{b})$ and $N(\frac{z}{v})$ are unoriented $2$-bridge
 knots or links.  
If $w \not\equiv \pm 1$ or if $U$
 is rational, 
then $\frac{t}{w} = \frac{xz - av'}{bv' - yz - ht}$ and $U = \frac{a}{b + ha}$, 
or $\frac{t}{w} = \frac{bz - av'}{ xv' - yz - ht}$ and $U = \frac{a}{x + ha}$, 
where $v'$ is any
integer such that $v'v^{\pm 1} = 1$ mod $z$.  
If $w \equiv \pm 1$ mod
$t$, then $t$ divides $z \mp a$.
\end{cor}


 If $w \not\equiv \pm 1$ or if $U$ is rational and if $N({\frac{a}{b}}) = N(2k)$, then $\frac{t}{w}
 = \frac{z - 2kv'}{v' - (z - 2kv')h}$ and $U = \frac{2k}{1 + 2kh}$.
If 
  $w\equiv\varepsilon=\pm 1$ mod $t$, then  $t$, $p$, and $pb-qa$ in part (b) of Theorem \ref{Thm:D_gM} are all factors of $z \mp a$.
Note that if $U$ is not rational in Theorem \ref{Thm:D_gM}, then both $|p| > 1$ and $|pb - qa| > 1$.  We will use this to show that the conclusion (b) of Theorem \ref{Thm:D_gM} does not occur under the assumptions of 
Theorem \ref{non-band6cat}, \ref{non-band8cat} or \ref{non-band10cat}.

\begin{proof}[Proof of Theorem \ref{non-band6cat}]
Suppose $k=3$ and $\frac{z}{v}=\frac{4mn-1}{2m}=\pm\frac{11}{2}$ or
 $\pm\frac{15}{4}$.  
If $U$ is rational, then Theorem \ref{non-band6cat} holds by Corollary \ref{Cor:D}.
Suppose $U$ is not rational.   
Then $|tp(p-6q)|=|z-6\varepsilon|=5,9,17$ or $21$.
But this is not possible since 
 $|t|$, $|p|$, and $|p - 6q|$ are all greater than 1.
 \end{proof}

\begin{proof}[Proof of Theorem \ref{non-band8cat}]
Suppose $k=4$ and $\frac{z}{v}=\frac{4mn-1}{2m}=\pm\frac{15}{2}$ or $\pm\frac{23}{4}$.
 If $U$ is rational, then Theorem \ref{non-band8cat} holds by Corollary \ref{Cor:D}.
Suppose $U$ is not rational.  
Then $|tp(p-8q)|=|z-8\varepsilon|=7,15,23$ or $31$.
But this is not possible since  $|t|$, $|p|$, and $|p - 8q|$ are all greater than 1.
\end{proof}

\begin{proof}[Proof of Theorem \ref{non-band10cat}]
Suppose $k=5$ and $\frac{z}{v}=\frac{4mn-1}{2m}=\pm\frac{19}{2}, \pm\frac{31}{4}$ or $\pm\frac{35}{6}$.
If $U$ is rational, then Theorem \ref{non-band10cat} holds by Corollary \ref{Cor:D}.
Suppose $U$ is not rational.  
Then $|tp(p-10q)|=|z-10\varepsilon|=9,21,25,29,41$ or $45$.  
Since  $|t|$, $|p|$, and $|p - 10q|$ are all greater than 1, the only
 possibility is that $|z-10\varepsilon|=45$ and $|p(p-10q)|=9$ or $15$.  
But there is no integer solution to $|p(p-10q)|=9$ or $15$ where $|p|$ and $|p - 10q|$ are both greater than one.
\end{proof}


The action of Xer recombination is believed to correspond to a $(-\frac{1}{3}, -\frac{4}{3})$ move.
By Theorem \ref{t-ratequiv}, a $(-\frac{1}{3}, -\frac{4}{3})$ move is equivalent to 
a $(0,\frac{9}{9l+5})$ move for any integer $l$.
Moreover, $N(U+0)=K_1$ and $N(U+\frac{9}{5})=K_2$ if and only if 
$N([U\circ (1,2,0)]+(-\frac{1}{3}))=K_1$ and $N([U\circ (1,2,0)]+(-\frac{4}{3}))=K_2$, 
since $N(U+0)=N([U\circ (1,2,0)\circ (-2,-1)]+0)=N([U\circ (1,2,0)]+[(0)\circ (-1,-2,0)])=N([U\circ (1,2,0)]+(-\frac{1}{3}))$ 
and $N(U+\frac{9}{5})=N(U+(-\frac{4}{3})\circ (2,1))=N([U\circ (1,2,0)]+(-\frac{4}{3}))$.
Then we obtain following corollaries from Theorem \ref{non-band6cat}, \ref{non-band8cat} and \ref{non-band10cat}.
\begin{cor}\label{-1/3to-4/36cat}
Suppose $N(U+(-\frac{1}{3}))=N(6)$.
\begin{enumerate}
\item If $N(U+(-\frac{4}{3}))=7_2$ $(\frac{z}{v}=\pm\frac{11}{2})$, then it has no solution.
\item If $N(U+(-\frac{4}{3}))=7_4$ $(\frac{z}{v}=\pm\frac{15}{4})$, then $U=(-\frac{1}{3})$.
\end{enumerate}
\end{cor}
\begin{cor}\label{-1/3to-4/38cat}
Suppose $N(U+(-\frac{1}{3}))=N(8)$.
\begin{enumerate}
\item If $N(U+(-\frac{4}{3}))=9_2$ $(\frac{z}{v}=\pm\frac{15}{2})$, it has no solution.
\item If $N(U+(-\frac{4}{3}))=9_5$ $(\frac{z}{v}=\pm\frac{23}{4})$, then $U=(-\frac{1}{5})$.
\end{enumerate}
\end{cor}
\begin{cor}\label{-1/3to-4/310cat}
Suppose $N(U+(-\frac{1}{3}))=N(10)$.
\begin{enumerate}
\item If $N(U+(-\frac{4}{3}))=11a247$ $(\frac{z}{v}=\pm\frac{19}{2})$, then it has no solution.
\item If $N(U+(-\frac{4}{3}))=11a343$ $(\frac{z}{v}=\pm\frac{31}{4})$, then $U=(-\frac{1}{7})$.
\item If $N(U+(-\frac{4}{3}))=11a363$ $(\frac{z}{v}=\pm\frac{35}{6})$, then it has no solution.
\end{enumerate}
\end{cor}


\section{Summary or Conclusion or Software}\label{summary}

Per theorems \cite{E, D_unor}, the system of tangle equations $N(U + B) = N(\frac{a}{b})$, $N(U + E) = N(\frac{z}{v})$ is easily solved when the $(B, E)$ move is equivalent to a $(0, \frac{t}{w})$ move where $|t| > 1$.  The case when $|t| = 1$ is much more difficult.  A few special subcases when  $|t| = 1$ can be handled using results in \cite{HS, MR2299739}.  Our theorem \ref{Thm:rts} handles the subcase when the move corresponds to a coherent banding and $N(\frac{a}{b}), N(\frac{z}{v}) \in \{N(2k), N(\frac{4mn - 1}{2m})\}$.  This subcase is particularly   biologically relevant since $N(\frac{4mn - 1}{2m})$ includes the family of twist knots, and we also applied it  to analyze the experimental results of Xer recombination acting  on $(2, 2k)$-torus links \cite{BSC}.  

The software TopoIce-R \cite{Darcy15072006} within Knotplot \cite{SchPhD} solves the system of tangle equations $N(U + B) = N(\frac{a}{b})$, $N(U + E) = N(\frac{z}{v})$ when $U$ is a generalized $M$-tangle (i.e., ambient isotopic to a sum of rational tangles) and $B$ and $E$ are rational tangles.  When the $(B, E)$ move is equivalent to a $(0, \frac{t}{w})$ move where $|t| > 1$, then 
 $U$ must be a   generalized $M$-tangle by Theorem \ref{E}.  However, other types of tangles can be solutions for $U$ when $|t| = 1$ \cite{D_xer}.  These solutions are not currently found by TopoICE-R. 
  Note that the solutions for $U$  in theorem \ref{Thm:rts} in which  $|t| = 1$ are generalized $M$-tangle.
 Hence for the cases in Theorem \ref{Thm:rts}, TopoIce-R finds all solutions to this system of tangle equations.  The solutions found by TopoICE-R correspond to performing surgery on a $(p, q)$ torus knot in the double branch cover of $N(\frac{a}{b})$.

\section*{Acknowledgments}
K.I. is partially supported by EPSRC Grant EP/H031367/1 to D.Buck.
K.S. is partially supported by KAKENHI 18540069 and 22540066.  
This research was supported in part by a grant from the Joint DMS/NIGMS
 Initiative to Support Research in the Area of Mathematical Biology (NSF 0800285) to I.D.

\baselineskip 4.8mm
   \bibliographystyle{amsplain}
   \bibliography{xerbib}

\end{document}